\documentclass{llncs}
\usepackage[utf8]{inputenc}
\usepackage{url,amsmath,amssymb,bussproofs}

\usepackage{xcolor}

\let\phi=\varphi

\newtheorem{problema}{Problem}

\newcommand{\Kcombinator}[0]{\mathsf{K}}

\newcommand{\rul}[1]{\mathrm{#1}}
\newcommand{\Dten}[0]{\rul{D10}}
\newcommand{\Dtenprime}[0]{\rul{D10^{\prime}}}
\newcommand{\Dtenstar}[0]{\rul{D10}^{*}}
\newcommand{\Deleven}[0]{\rul{D11}}
\newcommand{\Dtwelve}[0]{\rul{D12}}
\newcommand{\Dthirteen}[0]{\rul{D13}}
\newcommand{\Erule}[0]{\rul{E}}

\newcommand{\ruleset}[1]{\mathrm{#1}}
\newcommand{\Druleset}[0]{\ruleset{D}}
\newcommand{\Eruleset}[0]{\ruleset{E}}
\newcommand{\Erulesetstar}[0]{\ruleset{E}^{*}}
\newcommand{\Nruleset}[0]{\ruleset{N}}
\newcommand{\CLruleset}[0]{\ruleset{CL}}

\newcommand{\logic}[1]{\mathsf{#1}}
\newcommand{\CL}[0]{\logic{CL}}
\newcommand{\IL}[0]{\logic{IL}}
\newcommand{\N}{\logic{N}}
\newcommand{\LQ}{\logic{LQ}}
\newcommand{\Eprimevalid}[0]{\logic{E^{\prime}}}
\newcommand{\Stable}[0]{\logic{S}}

\newcommand{\valid}[2]{\vDash_{#1}{#2}}
\newcommand{\invalid}[2]{\nvDash_{#1}{#2}}
\newcommand{\Dvalid}[1]{\valid{\Druleset}{#1}}
\newcommand{\Evalid}[1]{\valid{\Eruleset}{#1}}
\newcommand{\CLvalid}[1]{\valid{\CLruleset}{#1}}
\newcommand{\Estarvalid}[1]{\valid{\Erulesetstar}{#1}}
\newcommand{\Nvalid}[1]{\valid{N}{#1}}

\newcommand{\Ninvalid}[1]{\invalid{N}{#1}}
\newcommand{\setofvalid}[1]{\mathcal{D}(#1)}

\newcommand{\WEM}[0]{\mathrm{WEM}}
\renewcommand{\implies}[2]{#1 \rightarrow #2}

\newcommand{\proponent}[0]{\mathbf{P}}
\newcommand{\opponent}[0]{\mathbf{O}}

\newcommand{\initialmove}[1]{0 & $\proponent$ & #1 & \emph{(initial move)}}
\newcommand{\prow}[4]{#1 & $\proponent$ & #2 & [#3,#4]}
\newcommand{\orow}[4]{#1 & $\opponent$ & #2 & [#3,#4]}

\usepackage{hyperref}
\hypersetup{
  bookmarks=true,
  pdftitle={Structure theory for dialogues},
  pdfkeywords={dialogue games, dialogical logic, classical logic, intuitionistic logic, semantics},
  pdfauthor={Jesse Alama},
  pdfsubject={An article presented at Computability in Europe 2010, Ponta Delgada, Portugal, 2010-07-03 (http://www.cie2010.uac.pt/)},
  colorlinks=true,
  linkcolor=black,
  citecolor=black,
}

\begin{document}
\title{Toward a structure theory for Lorenzen dialogue games}
\titlerunning{Structure theory for dialogues}
\author{Jesse Alama\thanks{Partially supported by the ESF
research project \emph{Dialogical Foundations of Semantics} within
the ESF Eurocores program \emph{LogICCC} (funded by the Portuguese
Science Foundation, FCT LogICCC/0001/2007).  The author wishes to thank Sara Uckelman for her comments and suggestions.}}
\tocauthor{Jesse Alama}
\institute{Theory and Logic Group{\\}Technical University of Vienna{\\}\email{alama@logic.at}}

\maketitle

\begin{abstract}
  Lorenzen dialogues provide a two-player game formalism that can
  characterize a variety of logics: each set $S$ of rules for such a
  game determines a set $\setofvalid{S}$ of formulas for which one of
  the players (the so-called Proponent) has a winning strategy, and
  the set $\setofvalid{S}$ can coincide with various logics, such as
  intuitionistic, classical, modal, connexive, and relevance logics.
  But the standard sets of rules employed for these games are often
  logically opaque and can involve subtle interactions among each
  other.  Moreover, $\setofvalid{S}$ can vary in unexpected ways with
  $S$; small changes in $S$, even logically well-motivated ones, can
  make $\setofvalid{S}$ logically unusual.  We pose the problem of
  providing a structure theory that could explain how $\setofvalid{S}$
  varies with $S$, and in particular, when $\setofvalid{S}$ is closed
  under modus ponens (and thus constitutes at least a minimal kind of
  logic).
\end{abstract}

\section{Introduction}

Lorenzen dialogue games~\cite{lorenzen1958} were offered as an
alternative game-theoretic formalism for intuitionistic logic (both
propositional and first-order).  The first player, Proponent
($\proponent$), lays down a logical formula and strives to
successfully respond to the assaults of Opponent.  The motion of the
game is determined by rules that depend on the structure of a formula
appearing in the game (which is always a subformula or, in the case of
a first-order game, an instance of, the initial formula played by
$\proponent$), as well as by rules that depend less on the form of
the formula at issue but rather concern the global structure of the
game and what kinds of roles can permissibly be played by $\proponent$
and Opponent (who are not merely dual to one another, as the players
often are in other logic games~\cite{sep-logic-games}).  Although
Felscher's equivalence theorem cleanly relates winning strategies of
Lorenzen games to intuitionistic validity, the rules for these games
are not entirely straightforward and indeed some of them appear to be
arbitrary.


Lorenz claimed that Lorenzen's dialogue games offer a new type of
semantics for intuitionistic logic and asserts the equivalence between
dialogical validity (defined in terms of winning strategies for the
$\proponent$) and intuitionistic derivability~\cite{lor9,lor10}. Lorenz's
proof contained some gaps, and later authors sought to fill these
gaps; a complete proof can be found in~\cite{felscher1985}.

Dialogue games are not restricted to intuitionistic logic.  By
modifying the rules of the game, the dialogue approach can also
provide a semantics for classical logic. The dialogical approach can
be adapted equally well to capture validity for other logics, such as
paraconsistent, connexive, modal and linear
logics~\cite{sep-dialogical-logic,ruck}.  All of these extensions of
Lorenzen's and Lorenz's initial formulation of dialogue games are
achieved by modifying the rules of the game while maintaining the
overall dialogical flavor.

The fact that there is no principled restriction on how the dialogical
rules can be modified naturally raises the question of when the set of
$S$-valid formulas, for a particular set $S$ of dialogical rules,
actually corresponds to a logic.  That is, we are interested in
identifying desirable properties of the set of $S$-valid formulas in
order to give it some logical sensibility.  One such desirable
property is that the set be closed under modus ponens: If $\phi$ and
$\phi \rightarrow \psi$ are $S$-dialogically valid, then so should
$\psi$ be.  We propose to call the problem of resolving whether a set
$S$ of rules for dialogue games satisfies this property the
\emph{composition problem} for $S$.

The structure of this paper is as follows.  The next section
introduces dialogue games and provides a few examples to make the
reader familiar with the basic definitions and notation.  Section 3
poses the \emph{composition problem}.  We generalize the problem and
motivate it from two perspectives of dialogues.  Section~\ref{sec:vary-dial-rules} presents
the results of three initial experiments that bear on the composition
problem.  Section~\ref{sec:nearly} describes a curious dialogical logic called $\N$.  Section~\ref{attempted} describes a failed (but apparently well-motivated) attempted dialogical characterization of the
intermediate logic $\LQ$ of weak excluded middle.  Section~\ref{sec:stable} takes on the problem of giving a dialogical characterization of stable logic (the intermediate logic in which the principle $\neg\neg p \rightarrow p$ holds for atoms $p$).  Section~\ref{sec:independent} motivates the
problem of giving independent rulesets and argues that it may be a
useful first step toward solving (instances of) the composition
problem.  Section~\ref{sec:conclusion} concludes and offers a few open problems for consideration.

\section{Dialogue games}

We largely follow Felscher's approach to dialogical
logic~\cite{felscher1985}.  For an overview of dialogical logic,
see~\cite{sep-dialogical-logic}.

We work with a propositional language; formulas are built from atoms
and $\neg$, $\vee$, $\wedge$, and $\rightarrow$.  In addition to
formulas, there are the three so-called \emph{symbolic attack}
expressions, $?$, $\wedge_{L}$, and $\wedge_{R}$, which are distinct
from all the formulas and connectives.  Together formulas and symbolic
attacks are called statements; they are what is asserted in a dialogue
game.

The rules governing dialogues are divided into two types.
\emph{Particle} rules say how statements can be attacked and defended
depending on their main connective.  \emph{Structural} rules define
what sequences of attacks and defenses count as dialogues.  Different
logics can be obtained by modifying either set of rules.

\begin{table}[t]
  \centering
  \setlength{\tabcolsep}{5pt}
  \begin{tabular}{c|c|c}
    \textbf{Assertion} & \textbf{Attack} & \textbf{Response}\\
    \hline
    $\phi \wedge \psi$       & $\wedge_{L}$ & $\phi$\\
                             & $\wedge_{R}$ & $\psi$\\
    $ \phi \vee \psi$        & $?$         & $\phi$ or $\psi$\\
    $ \phi \rightarrow \psi$ & $\phi$      & $\psi$\\
    $ \neg\phi$              & $\phi$      & ---
  \end{tabular}
\medskip
  \caption{Particle rules for dialogue games}
  \label{tab:particle-rules}
\end{table}

The standard particle rules are given in Table
\ref{tab:particle-rules}. According to the first row, there are two
possible attacks against a conjunction: The attacker specifies whether
the left or the right conjunct is to be defended, and the defender
then continues the game by asserting the specified conjunct.  The
second row says that there is one attack against a disjunction; the
defender then chooses which disjunct to assert.  The interpretation of
the third row is straightforward.  The fourth row says that there is
no way to defend against the attack against a negation; the only
appropriate ``defense'' against an attack on a negation $\neg\phi$ is
to continue the game with the new information $\phi$.

Further constraints on the development of a dialogue are given by the
structural rules.  In this paper we keep the particle rules fixed, but
we shall consider a few variations of the structural rules.

\begin{definition}
  Given a set $S$ of structural rules, an \emph{$S$-dialogue} for a
  formula $\phi$ is a dialogue commencing with $\phi$ that adheres to
  the rules of $S$. $\proponent$ \emph{wins} an $S$-dialogue if
  $\proponent$ made the last move in the dialogue and no moves are
  available for $\opponent$ by which the game could be extended.
\end{definition}

\begin{remark}
  According to this definition, if the dialogue \emph{can} go on, then
  neither player is said to win; the game proceeds as long as moves
  are available.  $\proponent$ wins by making a winning move; in other
  presentations of dialogue games such as
  Fermüller's~\cite{fermueller2003}, $\proponent$ wins when Opponent
  makes a \emph{losing} move.
\end{remark}

Winning strategies for dialogue games can be used to capture notions
of validity.

\begin{definition}
  For a set $S$ of dialogue rules and a formula $\phi$, the relation
  $\valid{S}{\phi}$ means that $\proponent$ has an $S$-winning strategy
  for $\phi$.  If $\nvDash_{S} \phi$, then we say that $\phi$ is
  $S$-invalid.  $\setofvalid{S}$ is the set $\{ \phi \colon
  \valid{S} \phi \}$.
\end{definition}
Note that, like usual proof-theoretic characterizations of validity,
dialogue validity is an existential notion, unlike the usual
model-theoretic notions of validity, which are universal notions.

We now consider two standard rule sets from the dialogue literature.
\begin{definition}
  The rule set $\Druleset$ is comprised of the following structural rules~\cite[p.~220]{felscher1985}:
  \begin{enumerate}
  \item[$(\Dten)$] $\proponent$ may assert an atomic formula only
    after it has been asserted by $\opponent$ before.
  \item[$(\mathrm{D11})$] When defending, only the most recent open
    attack (that is, attack against which no defense has yet been
    played) may be responded to.
  \item[$(\mathrm{D12})$] An attack may be answered at most once.
  \item[$(\Dthirteen)$] A $\proponent$-assertion may be attacked at most once.
  \end{enumerate}
\end{definition}


\begin{definition}
  The rule set $\mathrm{D}+\mathrm{E}$ is $\mathrm{D}$ plus the following rule:
  \begin{enumerate}
  \item[$(\mathrm{E})$] $\opponent$ can react only upon the immediately preceding $\proponent$-statement.
  \end{enumerate}
\end{definition}
\begin{definition}
  The rule set $\CLruleset$ is $\Eruleset - \{ \Deleven, \Dtwelve \}$.
\end{definition}
To give a sense of how these games proceed, let us look at a few
concrete examples of them.  In the following, note that we are working
with concrete formulas; ``$p$'' and ``$q$'' in the following are
concrete atomic formulas (atoms) and should not be read schematically
(indeed, if one were to substitute more complex formulas for $p$ and
$q$ in what follows, the examples would become incomplete in the sense
that they no longer necessarily represent wins or losses for
$\proponent$).

\begin{example}
  Let us consider a simple intuitionistic validity, the
  $\Kcombinator$-formula.  Table~\ref{k-formula-e-dialogue} lays out a
  concrete game for this formula.
\begin{table}[t]
  \centering
  \setlength{\tabcolsep}{5pt}
  \begin{tabular}{rcll}
    \initialmove{$\implies{p}{(\implies{q}{p})}$}\\
    \orow{1}{$p$}{A}{0}\\
    \prow{2}{$\implies{q}{p}$}{D}{1}\\
    \orow{3}{$q$}{A}{2}\\
    \prow{4}{$p$}{D}{3}
  \end{tabular}
\medskip
\caption{\label{k-formula-e-dialogue}An $\Eruleset$-dialogue for $\implies{p}{(\implies{q}{p})}$:}
\end{table}
This dialogue adheres to the $\Erule$-rule because $\opponent$ is
always responding to the immediately prior statement of $\proponent$.
Note that $\proponent$ is permitted to assert the atom $p$ at move $3$
because $\opponent$ already asserted it at move $1$.  $\proponent$
wins this game: $\opponent$ can make no further moves: the
$\Erule$-forces $\opponent$ to respond to move $4$ (in fact, it must
be attacked), but, in light of the particle rules, attacks on atoms
are not permitted.
\end{example}
\begin{example}
  Table~\ref{excluded-middle-e-dialogue} treats the classical law of the excluded middle, $p \vee \neg p$.
\begin{table}[t]
  \centering
  \setlength{\tabcolsep}{5pt}
  \begin{tabular}{rcll}
    \initialmove{$p \vee \neg p$}\\
    \orow{1}{?}{A}{0}\\
    \prow{2}{$\neg p$}{D}{1}\\
    \orow{3}{$p$}{A}{2}
  \end{tabular}
\medskip
\caption{\label{excluded-middle-e-dialogue}An $\Eruleset$-dialogue for excluded middle}
\end{table}
This short $\Eruleset$-dialogue (which, incidentally, is also a
$\Druleset$-dialogue) leads to a loss for $\proponent$: as in the
previous example, $\proponent$ is stuck.
\end{example}
\begin{example}
  Returning to excluded middle, let's see how the game goes when we
  change from intuitionistic to classical rules; see
  Table~\ref{excluded-middle-cl-dialogue}.
\begin{table}[t]
  \centering
  \setlength{\tabcolsep}{5pt}
  \begin{tabular}{rcll}
    \initialmove{$p \vee \neg p$}\\
    \orow{1}{?}{A}{0}\\
    \prow{2}{$\neg p$}{D}{1}\\
    \orow{3}{$p$}{A}{2}\\
    \prow{4}{$p$}{D}{1}
  \end{tabular}
\medskip
\caption{\label{excluded-middle-cl-dialogue}A $\CLruleset$-dialogue for excluded middle}
\end{table}
The difference between this dialogue, which $\proponent$ wins, and the
previous dialogue, which $\proponent$ lost, is that $\proponent$ can
now return to earlier attacks and defend against them in a new way.
The absence of rule $\Deleven$ from $\CLruleset$ makes the difference.
\end{example}
These examples should serve to give the reader a sense for how
dialogue games proceed, as one varies the rules.  Despite their apparent lack of logical
meaning, the rule sets $\Druleset$ and $\Eruleset$ have the following
property:
\begin{theorem}[Felscher]
  For all formulas $\phi$, the following are equivalent:
  \begin{itemize}
  \item $\phi$ is intuitionistically valid.
  \item $\Dvalid{\phi}$.
  \item $\Evalid{\phi}$.
  \end{itemize}
\end{theorem}
The proof goes by converting deductions in an intuitionistic sequent
calculus to $\mathrm{D}$-winning strategies (via tableaux), and vice
versa.  Moreover, the ruleset $\CLruleset$ has the following
significance:
\begin{theorem}[Felscher]
  For all formulas $\phi$, we have that $\phi$ is a classical tautology iff
  $\CLvalid{\phi}$.
\end{theorem}
In other words, dropping $\Deleven$ and $\Dtwelve$ from the ruleset
$\Eruleset$ moves us from intuitionistic to classical logic.

\section{The composition problem}
\label{sec:composition-problem}

One can view dialogue games in two (compatible) ways.  These games can
be a kind of rational dialogue between two players, or they can be
viewed as a kind of logical calculus.  In this section we shall
describe a problem about dialogues that bears on them no matter which
view one takes about dialogues.

The statement of the problem does not depend on which viewpoint we
adopt:
\begin{problema}[Composition]
  Given a set $S$ of structural rules, determine whether
  $\setofvalid{S}$ is closed under modus ponens, that is, whether it
  is true that $\phi \in \setofvalid{S}$ and $\implies{\phi}{\psi} \in
  \setofvalid{S}$ implies $\psi \in \setofvalid{S}$.
\end{problema}
One approach to the composition problem is to simply give positive
solutions for each ruleset $S$ that one is
interested in.  A more unifying problem is available, though:
\begin{problema}[Uniform composition]
  Give criteria for a set $S$ of dialogue rules (perhaps coming from
  some delimited class of rulesets) such that modus ponens is admissible for
  $\setofvalid{S}$.
\end{problema}
Instead of focusing on particular rulesets, the uniform composition
problem asks for \emph{criteria} for a ruleset which, if satisfied for
any ruleset $S$, ensure that we have a positive solution to the
composition problem for $S$.

The qualifier ``(perhaps coming from some delimited class of
rulesets)'' in the statement of the uniform composition problem
permits one to restrict the range of rulesets of interest (e.g., such
as those coming from various dialogical characterizations of modal
logic~\cite{sep-dialogical-logic}).  A totally general solution to the
uniform composition problem seems to be out of the question, putting
aside the question of what a ``dialogue rule'' in general is, which
makes it unclear over what the problem quantifies.

We now consider this problem from the two points of view about dialogues.

\subsection{Dialogues as rational interaction}
\label{sec:dial-as-rati}

If dialogues are to be for a (stylized) kind of rational interaction,
then one ought to have a criterion according to which one can say that
certain (sets of) dialogue rules support or undermine the rational
behavior of the players.

From results like Felscher's we can see that there must be a positive
solution to the composition problem: since intuitionistic logic is
actually a logic, if $\proponent$ has winning strategies for $\phi$ and
$\implies{\phi}{\psi}$, then $\proponent$ must have a winning strategy
for $\psi$.

Thus, by singling out the composition problem, we are not necessarily
raising a genuinely new problem about dialogue games, at least not in
all cases, where correspondence results are known, such as for
intuitionistic and classical first-order logic, modal logics, and so
forth.  Rather, we are proposing a problem with a change of emphasis:
rather than solving the composition problem as a corollary of
considerably stronger results, we raise the following challenge for
dialogue games: \emph{if a set of dialogue rules $S$ is supposed to
  actually be a coherent logic, we would like to have a direct proof
  of this fact; it should, ideally, be possible to give a direct
  solution of the composition problem for $S$ before a technically
  complex correspondence is established between the set of formulas
  for which $\proponent$ has a winning strategy and the set of known
  validities.}

Another way to approach the composition problem: if dialogue games
based on a set $S$ of dialogue rules are supposed to be an
\emph{autonomous} foundation for some kind of logic $L$, then it
should be possible to solve the composition problem for $S$ without
reference to whatever ``machinery'' for $L$ has been built up outside
of the dialogical approach.

What do we mean by ``rational''?  Various senses are available, for a
ruleset~$S$:
\begin{itemize}
\item An $S$-strategy for a formula $\phi$ should correspond to
  some conclusive reasoning for $\phi$;
\item if $\proponent$ has an $S$-winning strategy for $\phi$, then $\proponent$
  does not have an $S$-winning strategy for $\neg\phi$;
\item if $\proponent$ has an $S$-winning strategy for $\phi$, then
  Opponent does not also have an $S$-winning strategy for $\phi$;
\item if $\proponent$ has $S$-winning strategies for $\phi$ and
  $\implies{\phi}{\psi}$, then $\proponent$ has an $S$-winning strategy
  for $\psi$.
\end{itemize}
The fourth explication of $S$-strategy rationality is simply the same as
the composition problem for $S$.

We can further distinguish two loci of rationality: games and strategies.
\begin{definition}[Game rationality] A ruleset $S$ is \emph{game-rational} if the
  development of $S$-dialogue games should have the form of a rational
  conversation between two opposing players.
\end{definition}

\begin{definition}[Strategy rationality]
  A ruleset $S$ is \emph{strategy-rational} if $S$-winning strategies
  constitute some kind of rational argument.
\end{definition}

One way to deflate the composition problem is to acknowledge that
dialogue games are not in fact supposed to be an autonomous foundation
of capturing validity in a logic.  (One might even wonder what it
means to be an \emph{autonomous} foundation for a logic.)  Or we are
to drop the requirement about game-rationality or strategy-rationality
for dialogue rulesets.  And it would seem that neither of these
desiderata can really be abandoned, if one wishes to see dialogue
games as more than a mere calculus and having something to do with
``rational dialogue''.  It seems we lack a compelling account of the
rationality of dialogues, in the sense that we lack a defense of
certain sets of dialogue rules over others.\footnote{Woods has
  highlighted another problem concerning the rationality of dialogue
  games, different from ours, which is related to the problem of
  logical omniscience~\cite{woods1988}.  Walton also sketches some
  problems of rationality in dialogues~\cite{walton1985}.}  If one
views dialogues as simply alternative calculi for working with
different logics, then one might still be persuaded by our call for
``direct'' solutions to the composition problem.  This point of view
is taken up in the next subsection.

\subsection{Dialogues as calculi}
\label{sec:dialogues-as-calculi}

Apart from treating dialogue games as a stylized debate or rational
interaction between two opposing players, one can view these games as
a logical calculus on a par with other formalisms for proofs such as
Hilbert-style, natural deduction, tableaux, or sequent calculus.
(These two points of view are, of course, compatible.)  From this
point of view, the composition problem for a ruleset $S$ is the
problem of showing that modus ponens is an admissible rule of
inference for $\setofvalid{S}$, the set of all formulas $\phi$ for
which $\proponent$ has an $S$-winning strategy for $\phi$.

One way to view the problem is that we have a handful of positive
results: for a certain very limited number of dialogue rulesets, we
know about them that they correspond to certain logics (and hence
positively solve their associated composition problems).  We may view
these positive results as local maxima in a space populated by logics
and non-logics alike.  We wish to understand what happens when we step
away from these local maxima in this space.  Certainly, some
curiosities will result (see section~4.1 for an example).  The
perspective behind the uniform composition problem is to embrace these
non-maxima (or perhaps even discovering new maxima) in the hopes of
understanding the whole space: let us shift from a (very) discrete
point of view to a ``continuous'' point of view, to see what the
dialogical space is like.

One can evidently point to theorems such as Felscher's to dispense
with the composition problem for the rulesets $\Druleset$ and
$\Eruleset$.  However, Felscher's theorem does not, prima facie, solve
the uniform composition problem.  Some positive results bearing on the
uniform composition problem are those of
Fermüller~\cite{fermueller2003}, who, using so-called parallel
dialogue games, gives dialogical characterizations of a variety of
intermediate logics.  We shall return to Fermüller's results later, in
section \ref{attempted}.

We are also interested in the question of to what extent dialogue
games actually offer a fine-grained division of different kinds of
logics.  If it turns out that only a handful of sets of dialogue rules
are adequate for the purpose of generating a logic (i.e., for
capturing some minimally rational meaning of a dialogue game), then
this needs to be explained.  That is, if it turns out that there is
something unique about the standard sets of dialogue rules that have
heretofore been investigated, then this serves as a critical point for
the dialogue approach, because it shows that its apparent opportunity
for logical generality is in fact highly constrained and tightly
delimited.


\subsection{Direct solutions to the composition problem}\label{sec:direct-solutions}

A positive solution to the composition problem for a set of dialogue
rules tells us that our rule set gives rise to a logic, at least in a
weak sense of the term ``logic''. Of course, we are likely not
interested in the case where all formulas are valid, in which case the
composition problem is trivially solved in the affirmative.

We have indicated that we prefer direct solutions to the composition
problem.  We can certainly bring to bear whatever means we have toward
establishing significant properties of a dialogical logic.  But if a
``direct'' solution to a problem is available, it seems reasonable to
provide one alongside whatever other methods we have.  The problem is
simply that we wish to have multiple proofs.  Whether a direct proof
that operates on winning strategies is ``the same'' as a proof of the
same result using some other methods is not always clear.  Even if a
positive solution to the composition problem is ``really'' a disguised
version of cut elimination, there may still be value in working
directly with strategies rather with, say, sequent calculi
derivations, since we don't need to first do the work of showing that
the sequent calculi really captures the strategies.

Direct solutions to dialogical problems may be the only solutions, if
one is exploring a logic whose relation to other, differently
characterized logics is unknown.  A positive, direct solution to the
composition problem for such a logic is given in~\cite{AU} (see also Section~\ref{sec:vary-dial-rules}).

To illustrate further what we have in mind by ``direct'' solutions to
dialogical problems, we now give a positive solution to the
composition problem for a dialogical characterization of classical
propositional logic.  We will first present a set $\CLruleset$ of dialogue
rules, and then we will show that $\CLvalid{\phi \rightarrow \psi}$
and $\CLvalid{\phi}$ implies $\CLvalid{\psi}$ by working with
$\CLruleset$-winning strategies.  We will not show that $\CLruleset$
captures $\CL$; see~\cite{felscher1985}.

\begin{lemma}\label{no-atom-cl-valid}
  No atom is $\CLruleset$-valid.
\end{lemma}
\begin{proof}
  The $\CLruleset$ does not permit a game to even get started with the
  assertion by $\proponent$ of an atom.
\end{proof}
Such a result obviously holds for any set of dialogue rules that
contains $\Dten$.

The next lemma is a kind of consistency result of classical logic,
construed dialogically.
\begin{lemma}[No explosion]\label{no-explosion}
  There is no $\CLruleset$-valid formula $\phi$ with the property that
  $\CLvalid{\phi \rightarrow \psi}$ for all formulas $\psi$.
\end{lemma}
Such a formula gives rise to an ``explosion'' in the sense that it
entails (in the object language) all formulas.  If there were such a
formula $\phi$, we would have, for example, that $\CLvalid{\phi
  \rightarrow p}$, even for atoms $p$ that do not occur in $\phi$.
Such a case is clearly untenable.  We have not yet been able to find a
direct proof of this lemma, but it does seem to us to be an important
step toward a direct proof of the composition problem for
$\CLruleset$.  (Note that, by closure of $\CLruleset$-valid formulas
under modus ponens, such an ``explosion'' formula does not exist,
since it would imply, as we said, that all atoms would be valid, which
of course violates Lemma~\ref{no-atom-cl-valid}.)  Such a problem can
clearly be solved quite easily using the truth table notion of
classical validity.  Less easy is a proof-theoretic solution to the
problem using a sequent calculi adequate for classical logic; the
solution apparently requires cut elimination.


\begin{theorem}[Attack-first]
  If $\CLvalid{\phi}$, then there is a $\CLruleset$-winning strategy
  for $\phi$ in which $\proponent$'s defenses are delayed as far as
  possible.
\end{theorem}
\begin{proof}
  The idea is that we consider $\CLruleset$-winning strategies $\tau$
  that have the property that, for each branch $b$ of $\tau$, and each
  $\proponent$ move $m$ of $b$, if $m$ is a defense, then at $m$ it is
  not possible for $\proponent$ to attack.  That is, we consider
  $\CLruleset$-winning strategies where $\proponent$ must defend; if
  $\proponent$ can attack, then he does.

  The existence of such $\CLruleset$-winning strategies is clear.  If
  there is a $\CLruleset$-winning strategy for $\phi$, but only one,
  then it satisfies the attack-first condition because $\proponent$
  has no alternatives available to him.  For a $\CLruleset$-valid
  formula $\phi$, there could even be multiple such strategies.
\end{proof}
We can refine the attack-first strategy further by requiring that, if
no attacks but multiple defenses are available for $\proponent$, then
we require that $\proponent$ defend against the most recent attack.
\begin{theorem}[Attack-first-defend-most-recent]
  IF $\CLvalid{\phi}$, then there exists a $\CLruleset$-winning
  strategy in which $\proponent$'s defenses are delayed as far as
  possible, and in which, if multiple defenses are possible for
  $\proponent$, then the defense against the most recent attack is
  chosen.
\end{theorem}
The existence of such strategies for $\CLruleset$-valid formulas is
again clear.

Note that, unlike proofs of analogous results via cut elimination,
these ``normal forms'' for dialogues do not require the definition of
a reduction relation and a proof that it is normalizing; the existence
of $\CLruleset$-winning strategies adhering to these conditions is
clear.

We have so far not been able to find a ``direct'' proof of
Lemma~\ref{no-explosion}.  Such a result must hold, since, thanks to
Felscher's and other dialogical characterization of classical logic
(e.g.,~\cite{SU}), we have $\CLvalid{\phi}$ iff $\phi$ is a classical
tautology.  From the perspective of truth tables, such a statement
clearly holds: a $\phi$ with this property would be a contradiction
such as $\perp$ or $p \wedge \neg p$, but such statements are not
valid.  The following is an outline of a proof using ``direct''
methods, using the ideas developed so far.
\begin{theorem}
  If $\CLvalid{\phi}$ and $\CLvalid{\phi \rightarrow \psi}$, then $\CLvalid{\psi}$.
\end{theorem}
\begin{proof}[Sketch]
  Consider a $\CLruleset$-winning strategy $d$ for $\phi \rightarrow
  \psi$.  It begins with the assertion by $\proponent$ of $\phi
  \rightarrow \psi$, then an attack by $\opponent$ on this
  implication, asserting $\phi$.  The beginning is shown in
  Table~\ref{winning-strat-for-implication}.
  \begin{table}[t]
    \centering
    \setlength{\tabcolsep}{5pt}
    \begin{tabular}{rcll}
      \initialmove{$\phi \rightarrow \psi$}\\
      \orow{1}{$\phi$}{A}{0}\\
    \end{tabular}
    \medskip
    \caption{\label{winning-strat-for-implication}Beginning of every $\CLruleset$-winning strategy for $\phi \rightarrow \psi$}
  \end{table}
  We do not know what the next step is; $\proponent$ could attack
  $\phi$ or defend against the initial assertion of $\psi$.  There are
  three possibilities:
  \begin{itemize}
  \item $\proponent$ never defends against the initial attack on $\phi
    \rightarrow \psi$.  In this case, evidently it makes no difference
    what $\psi$ is, so by simply changing the first step of $d$ from
    $\phi \rightarrow \psi$ to $\phi \rightarrow \chi$, we have a
    $\CLruleset$-winning strategy for $\phi \rightarrow \psi$, no
    matter what $\psi$ is.  $\phi$ would thus a counterexample to Lemma~\ref{no-explosion}.
  \item $\proponent$ never attacks $\opponent$'s assertion of $\phi$.
    Then $d$ is evidently already a $\CL$-winning strategy for $\psi$,
    provided we simply delete the initial two moves.
  \item $\proponent$ does defend against the initial attack on $\phi
    \rightarrow \psi$.  This is the general case, and likely the most
    difficult.  The main idea is to look for a suitable rewriting, or
    normal form, of the strategy into one from which we can extract a
    $\CLruleset$-winning strategy for $\phi \rightarrow \psi$.  It
    seems plausible that the \emph{defend last} normal form defined
    earlier would be helpful.  By adhering to that normal form, we
    defer $\proponent$'s defense against the initial attack as long as
    possible, forcing $\opponent$ to make the greatest number of
    commitments (viz., assert the most atoms) before coming to the
    defense against the initial attack.
  \end{itemize}
\end{proof}
We have targeted $\CL$ here and a dialogical characterization
($\CLruleset$) of it because $\CLruleset$ is somewhat
relaxed compared to the rulesets $\Druleset$ and
$\Eruleset$, which are known to be adequate for intuitionistic logic.  That $\IL$
is closed under modus ponens is, of course, obvious.  It is not clear
to us whether the strategy normal forms that we have proposed (namely,
the attack-first and its refinement, attack-first-defend-latest forms)
have the same significance in the presence of rules $\Deleven$ and
$\Dtwelve$ as they do when these two rules are missing (which is the
case for the standard dialogical characterization of $\CL$).

\section{Varying dialogue rules}
\label{sec:vary-dial-rules}

To illustrate our approach, let us look at some examples where one
varies the rulesets.  This section reports on three such experiments.

\subsection{Nearly classical logic}\label{sec:nearly}

We have stated earlier that Felscher's theorem shows the
correspondence between the $\Druleset$ and $\Eruleset$ rulesets and
intuitionistic logic $\IL$.  Since $\IL$ is closed under modus
ponens, Felscher's theorem implies that
$\setofvalid{D}$ and $\setofvalid{E}$ are likewise both closed under modus ponens.
It is also known that, if one drops Felscher's $\Deleven$
and $\Dtwelve$ from $\Druleset$, but adds rule $\Erule$, one obtains a
dialogical characterization of classical logic $\CL$.

Is rule $\Erule$ necessary for modus ponens?
\begin{definition}
  Let $\Nruleset$ be $\Druleset - \{ \Deleven, \Dtwelve \}$, and let $\N$ be
  the set of $\Nruleset$-valid formulas.
\end{definition}
Since dropping the $\Erule$ makes no difference when passing from
$\Eruleset$ to $\Druleset$, it is true that closure under modus ponens
is preserved if one drops $\Erule$ from
$\Druleset - \{ \Deleven, \Dtwelve \} \cup \{ \Erule \}$ (which
dialogically captures $\CL$)?  More simply, is $\N$ closed under modus
ponens?

The answer, curiously, is that $\N$ is closed under modus ponens but
not under uniform substitution.  The following
necessary conditions govern $\N$'s valid implications:
\begin{theorem}\label{n-characterization-theorem}
  If $\Nvalid{\phi\rightarrow\psi}$, then either
\begin{enumerate}
\item $\Nvalid\psi$,
\item $\phi$ is atomic, or
\item $\phi$ is a negation.
\end{enumerate}
\end{theorem}
(For details, see~\cite{AU}.)  Using
Theorem~\ref{n-characterization-theorem}, many failures of uniform
substitution for $\N$ can be produced.  We have, for example, that
$\Nvalid{p \rightarrow \neg\neg p}$ (this can be shown by
calculation), but $\Ninvalid{(p \wedge p) \rightarrow \neg\neg(p
  \wedge p)}$, because the antecedent meets none of the necessary
conditions listed in Theorem~\ref{n-characterization-theorem}.  (That
$\Ninvalid{\neg\neg(p \wedge p)}$ can be shown by calculation.)

Adding rule $\Erule$ to $\Nruleset$ restores uniform substitution (and
maintains closure under modus ponens), so despite appearances, there
must be something about rule $\Erule$ intimately tied to uniform
substitution.

The presence of the $\Erule$ could be regarded as a mere technical
necessity for establishing a correspondence between existence of
winning strategies for dialogue games and some notion of logical
validity, characterized without using dialogues.  The $\Erule$ rule
has no obvious correspondence with everyday dialogue; even if one were
inclined to adopt some kind of regimentation, the $\Erule$ appears to
be a rather strong constraint.  A better understanding of its
eliminability is wanted.  Results such as the curious $\N$ show that,
at least in one well-known setting (classical dialogue games), $\Erule$
cannot be entirely dropped, while in at least one other setting
(intuitionistic logic) it can be dropped.  One problem would be to
find some relaxation of $\Erule$ that suffices for $\CL$.  It seems
fruitful us to investigate the precise conditions under which
repetitions are permitted.  Already some work has been done in this
direction (see~\cite{krabbe85}) for intuitionistic logic.

\subsection{An attempted dialogical characterization of the logic $\LQ$}\label{attempted}

In a Hilbert-style calculus for propositional logic, one can start
with intuitionistic logic and obtains classical logic by adding
additional axioms, such as Peirce's formula, excluded middle, or
double negation elimination (the precise details depend on which
propositional signature one is interested in).

With dialogues, one moves from intuitionistic to classical logic not
by adding but by removing dialogue rules.  In the dialogical setting,
classical logic can be obtained by relaxing the dialogue rules for
intuitionistic logic.\footnote{The precise claim is that one can
  obtain a dialogical characterization of classical logic by removing
  $\Deleven$ and $\Dtwelve$ from the ruleset $\Eruleset$.  We say
  ``can be obtained'' rather than ``is obtained'' because, depending
  on which ruleset one chooses for intuitionistic logic, our claim is
  false: the ruleset $\Nruleset$ and the set $\N$ that it generates
  shows that we do \emph{not} obtain classical logic by simply
  dropping $\Deleven$ and $\Dtwelve$ from the ruleset $\Druleset$.}
One might then naturally wonder if one can give dialogical
characterizations of intermediate logics (i.e., propositional logics
between $\IL$ and $\CL$) by adding dialogue rules to the ruleset
$\CLruleset$.

One natural experiment would be to try to capture a ``simple''
intermediate logic, such as Jankov's logic
$\LQ$~\cite{jankov1968,tvd}, which is $\IL$ together with the
principle of weak excluded middle (WEM), $\neg p \vee \neg\neg p$.
This principle is obviously classically valid but it is independent of
$\IL$ (one can see this using Kripke models).

Fermüller has given a dialogical characterization of $\LQ$ (and other
logics) with the help of parallel dialogue
games~\cite{fermueller2003}.  Fermüller matches winning strategies for
parallel dialogue games with derivability in a calculus based on
hypersequents due to Ciabattoni et
al.~\cite{ciabattoni-gabbay-olivetti1999}.  Fermüller's parallel
dialogue games diverge from the ``sequential'' games employed in this
paper.

Despite Fermüller's solution, one might still seek out a
``sequential'' characterization of $\LQ$, perhaps employing a
non-hypersequent formulation of $\LQ$, such as
Hosoi's~\cite{hosoi1988}.  Ideally, one would seek an intuitive,
self-contained addition or modification to some known ruleset, such as
the $\Eruleset$, that would characterize $\LQ$.  A first step would be
to find such a modification according to which $\neg p \vee \neg\neg
p$ is valid.

To motivate the new dialogue rule that will be introduced soon, let us
consider the $\Eruleset$-dialogue game for $\WEM$; see
Table~\ref{wem-plays}.
\begin{table}[t]
  \centering
  \setlength{\tabcolsep}{5pt}
  \begin{tabular}{rcll}
    \initialmove{$\neg p \vee \neg\neg p$}\\
    \orow{1}{?}{A}{0}\\
    \prow{2}{$\neg p$}{D}{1}\\
    \orow{3}{$p$}{A}{2}\\
  \end{tabular}
  \begin{tabular}{rcll}
    \initialmove{$\neg p \vee \neg\neg p$}\\
    \orow{1}{?}{A}{0}\\
    \prow{2}{$\neg\neg p$}{D}{1}\\
    \orow{3}{$\neg p$}{A}{2}\\
  \end{tabular}
\medskip
\caption{\label{wem-plays}Two losing plays for $\proponent$ in the $\Eruleset$-dialogue for weak excluded middle}
\end{table}
The two $\Eruleset$-dialogues for $\neg p \vee \neg\neg p$ of
Table~\ref{wem-plays} show that $\proponent$ loses quickly no matter
whether the initial attack is defended by asserting $\neg p$ or
$\neg\neg p$.  Since WEM is not intuitionistically valid, by
Felscher's theorem $\proponent$ does not have an $\Eruleset$-winning
strategy for it.  Indeed, the above two games, diverging at move $2$,
together make up all possible ways the game could go; $\proponent$
loses in both.  The obstacle seems to be $\Dten$, which blocks
$\proponent$ from asserting the atom $p$ before Opponent has conceded it.
We can see this by comparing the two $\Eruleset$-dialogues with how
the game goes when playing the ruleset $\CLruleset$ for classical
logic.
\begin{table}[t]
  \centering
  \setlength{\tabcolsep}{5pt}
  \begin{tabular}{rcll}
    \initialmove{$\neg p \vee \neg\neg p$}\\
    \orow{1}{?}{A}{0}\\
    \prow{2}{$\neg p$}{D}{1}\\
    \orow{3}{$p$}{A}{2}\\
    \prow{4}{$\neg\neg p$}{D}{1}\\
    \orow{5}{$\neg p$}{A}{4}\\
    \prow{6}{$p$}{A}{5}
  \end{tabular}
\medskip
\caption{\label{wem-win}A winning play for $\proponent$ in the $\CLruleset$-dialogue for weak excluded middle}
\end{table}
In the ruleset $\CLruleset$, $\proponent$ can return to earlier attack
and defend against them, unlike in the $\Druleset$ and $\Eruleset$
rulesets, in which multiple defenses are ruled out.  $\proponent$'s
ability to repeat earlier defenses makes all the difference, because
he can defend in move $4$, \emph{in a different way}, using
$\opponent$'s ``concession'' of the atom $p$ in move $3$.  (The game
of Table~\ref{wem-win} is in fact a winning strategy for $\WEM$.)

We require a set $S$ of dialogue rules ``between'' the ruleset
$\Eruleset$ and $\CLruleset$.  $\proponent$'s ability to return to
earlier defenses seems to be rather too strong.  Let us consider the
following modified form of $\Dten$:
\begin{itemize}
\item[$(\Dtenstar)$] $\proponent$ may assert an atom $p$ only if $\opponent$ has
  asserted either $p$ or $\neg p$ before.
\end{itemize}
Let $\Erulesetstar$ be $E$ except with $\Dtenstar$ instead of $\Dten$.
The idea is that $\WEM$ is a kind of excluded middle, but only for
\emph{negative} statements.  We modify $\Dten$ according to this
intuition: once $\opponent$ reveals some negative information (i.e.,
concedes a negated atom), $\proponent$ is permitted to proceed with
this information as though it were positive.  Table~\ref{estar-win} is
a calculation showing that at least one instance of
$\Estarvalid{\WEM}$ is valid, and proceeds in an intuitive way (from
the perspective of $\LQ$):
\begin{table}[t]
  \centering
  \setlength{\tabcolsep}{5pt}
  \begin{tabular}{rcll}
    \initialmove{$\neg p \vee \neg\neg p$}\\
    \orow{1}{?}{A}{0}\\
    \prow{2}{$\neg\neg p$}{D}{1}\\
    \orow{3}{$\neg p$}{A}{2}\\
    \prow{4}{$p$}{A}{3}
  \end{tabular}
\medskip
\caption{\label{estar-win}A winning play for $\proponent$ in the $\Erulesetstar$-dialogue for weak excluded middle}
\end{table}
But this rule goes overboard: we have not captured $\LQ$ but something
else, because the formula $\neg p \rightarrow p$ is
$\Erulesetstar$-valid.  This can easily be  seen: $\opponent$'s unique
opening move is to assert $\neg p$, and now $\proponent$ has a unique
response: to assert $p$, winning the game.

The lesson of this failure to capture the logic $\LQ$ using dialogues
is that we had a well-motivated modification to a basic dialogue rule,
but the consequences of adopting this rule were that unacceptable
formulas became valid.  Ideally, we would be able to appeal to a
structure theory that would explain the precise force of rule $\Dten$,
which would inform us ``in advance'' of what would happen if we were
to modify (or drop) it.

\subsection{Characterizing stable logic dialogically}\label{sec:stable}

Stable logic $\Stable$ is the intermediate logic axiomatized by the stability principle
\[
\neg\neg p \rightarrow p
\]

for atoms $p$.  The stability principle is not provable in intuitionistic logic.  This is intuitively clear by considering the Brouwer-Heyting-Kolmogorov interpretation, but it may be definitively shown by, e.g., a suitable Kripke model.  Although it has the flavor of being ``inherently classical'', stable logic is in fact strictly weaker than classical logic.  To obtain classical logic, it suffices to add the stability principle to Jankov's $\LQ$ discussed in Section~\ref{attempted}.  From the standpoint of the program considered in this paper, it is natural to ask how one can give a dialogical characterization of stable logic.  We found in Section~\ref{attempted} that, when we translated our semantic intuition of the principle of weak excluded middle into the dialogical context, the naive attempt failed.  In the case of stable logic, though, one's semantic intuition can be easily expressed dialogically.

To get started, let us consider, from the dialogical perspective, why stability is not provable.  Let us take the $\Druleset$ rules; see Table~\ref{stable-non-win}.  The game can develop in only one way.  $\proponent$ would like to assert $p$, and since $\opponent$ has also asserted it (move~$3$), so rule $\Dten$ wouldn't be violated.  $\proponent$ cannot because he cannot defend against the

\begin{table}[t]
  \centering
  \setlength{\tabcolsep}{5pt}
  \begin{tabular}{rcll}
    \initialmove{$\neg\neg p \rightarrow p$}\\
    \orow{1}{$\neg\neg p$}{A}{0}\\
    \prow{2}{$\neg p$}{A}{1}\\
    \orow{3}{$p$}{A}{2}\\
  \end{tabular}
\medskip
\caption{\label{stable-non-win}A non-winning play for $\proponent$ in the $\Druleset$-dialogue for the stability principle}
\end{table}

Already in $\IL$, one can prove $p \rightarrow \neg\neg p$.  Adding the stability principle, we have $p \leftrightarrow \neg\neg p$, so $p$ and $\neg\neg p$ are, as it were, on a par with one another.  Consider now the following new structural rule:

\begin{itemize}
\item[$(\Dtenprime)$] $\proponent$ may assert an atom $p$ only if $\opponent$ has
  asserted either $p$ or $\neg\neg p$ before.
\end{itemize}

This is reminiscent of our failed dialogical characterization of Jankov's $\LQ$ using $\Dtenstar$.  Here, though, instead of ``semantically identifying'' $p$ and $\neg p$ (which, in hindsight, is the root of the failure), we semantically identify $p$ and its double negation $\neg\neg p$.  Referring to the non-winning play in Table~\ref{stable-non-win}, it is clear that, were $\Dtenprime$ in effect rather than $\Dten$, the game would be ``short-circuited'' because, once $\opponent$ asserts $\neg\neg p$, $\proponent$ can pounce and assert $p$ in defense of the initial attack; $\opponent$ then cannot reply.

Of course, it is quite possible that, these positive signs notwithstanding, the adoption of our modified $\Dten$ has unacceptable consequences, as we saw in the preceding section when $\neg p \rightarrow p$ became dialogically valid.



Interestingly, the motivation behind the formalization of stable logic as presented in Negri and von~Plato~\cite{negri2001structural} is to give a sequent characterization of the familiar principle of indirect proof (``if from $\neg p$ one can derive a contradiction, then $p$ is provable''), here the principle is, roughly, the ``semantic identification'' of $p$ and $\neg\neg p$.

\begin{definition}
Let $\Eprimevalid$ be the set of $\Eruleset \setminus \{ \Dten \} \cup \{ \Dtenprime \}$-valid formulas.
\end{definition}

The aim is to show that $\Eprimevalid$ equals $\Stable$.  We are so far able to show part of this (see Theorem~\ref{stab-soundness}).  The following lemmas show that our proposal stands a chance of dialogically capturing $\Stable$.

\begin{lemma}\label{lemma:lem-1}
If $\phi$ is $\Eprimevalid$, then there exist atoms $p_{1}$, \dots, $p_{n}$ in $\phi$ such that
\[
\left ( \bigwedge_{1 \leq i \leq n} \neg\neg p_{i} \rightarrow p_{i} \right ) \rightarrow \phi
\]
belongs to $\IL$.
\end{lemma}
\begin{proof}
A move in dialogue games is not labeled by what rule justifies the move (unlike, say, sequent calculi or natural deduction).  The idea is that all structural rules govern all moves.  Nonetheless, let us nonetheless call an ``application'' of $\Dtenprime$ a move by $\proponent$ that would be impossible if the standard $\Dten$ were in effect rather than $\Dtenprime$.  Applications of $\Dtenprime$ thus look, schematically, like this:
\begin{table}[t]
  \centering
  \setlength{\tabcolsep}{5pt}
  \begin{tabular}{rcll}
    \orow{$m$}{$\neg\neg p$}{-}{-}\\
    \vdots\\
    \prow{$n$}{$p$}{-}{-}\\
  \end{tabular}
\end{table}
The idea is, for each application of $\Dtenprime$, to adjust the strategy---in fact, even the initial statement by $\proponent$---so that these applications are ``eliminated'' in the sense that the same sequence of moves goes through if $\Dtenprime$ were replaced the more strict $\Dten$.  Such rewrites are effected as follows:
\begin{itemize}
\item Before the initial statement of $\phi$ by $\proponent$, insert the following moves:
\begin{table}[t]
  \centering
  \setlength{\tabcolsep}{5pt}
  \begin{tabular}{rcll}
    \initialmove{$(\neg\neg p \rightarrow p) \rightarrow \phi$}\\
    \orow{1}{$\neg\neg p \rightarrow p$}{A}{0}\\
    \prow{2}{$\neg\neg p$}{A}{1}\\
    \orow{3}{$\neg p$}{A}{2}\\
    \orow{4}{$\phi$}{D}{1}
    \vdots
  \end{tabular}
\medskip
\caption{Eliminating an application of $\Dtenprime$ by explicitly postulating an instance of stability}
\end{table}
(Of course, the addition of new moves to the beginning of the game needs to be accompanied by relabeling any references.  Thus, a reference to move $3$ should now refer to move $7$, etc.)
\item Replace the application of $\Dtenprime$ as in Table~\ref{tab:doit}:
\begin{table}[t]
  \centering
  \setlength{\tabcolsep}{5pt}
  \begin{tabular}{rcll}
    \orow{$m$}{$\neg\neg p$}{-}{-}\\
    \vdots\\
    \prow{$n$}{$\neg p$}{A}{$m$}\\
    \orow{$n + 1$}{$p$}{A}{$n$}\\
    \prow{$n + 2$}{$p$}{D}{$1$}
  \end{tabular}\caption{\label{tab:doit}Eliminating an application of $\Dtenprime$}
\end{table}
\item Repeat until there are no more applications of $\Dtenprime$.
\end{itemize}
The effect of these repeated rewrites is that any ``exploit'' by $\proponent$ of the extra freedom granted by the relaxed $\Dtenprime$ gets turned into an extra
\end{proof}

The atoms claimed to exist in Lemma~\ref{lemma:lem-1} come from the winning strategy that witnesses that $\phi \in \Eprimevalid$.  It may be possible that, depending on the strategy, we get a different set of atoms.  In any case, adding stability for more atoms can't hurt (weakening is clearly acceptable).

\begin{lemma}\label{lemma:lem-2}
If $\phi \in \Eprimevalid$, then
\[
\left ( \bigwedge_{1 \leq i \leq n} \neg\neg p_{i} \rightarrow p_{i} \right ) \rightarrow \phi
\]
belongs to $\IL$, where $p_{1}$, \dots, $p_{n}$ lists all the atoms occurring in $\phi$.
\end{lemma}

Lemma~\ref{lemma:lem-2} gives a dialogical characterization of a familiar fact about stable logic (see, e.g.,~\cite[Ch.~3]{troelstra-schwichtenberg}), namely that one can ``reduce'' some intermediate logics to $\IL$ provided one explicitly postulates salient features of the intermediate logic (for the case of $\LQ$, see Hosoi~\cite{hosoi1988}, where the fact that explicitly postulate instances of the principle of weak excluded middle, playing the same role there as the stability principle does for us here).  Lemma~\ref{lemma:lem-2} combined with the fact that stable logic includes all instances of the stability principle yields:

\begin{lemma}\label{lemma:lem-3}
If
\[
\left ( \bigwedge_{1 \leq i \leq n} \neg\neg p_{i} \rightarrow p_{i} \right ) \rightarrow \phi
\]
is in $\IL$, then $\phi \in \Stable$.
\end{lemma}

Lemmas~\ref{lemma:lem-1} and~\ref{lemma:lem-3} yield:

\begin{theorem}\label{stab-soundness}
If $\phi \in \Eprimevalid$, then $\phi \in \Stable$.
\end{theorem}

Evidently, if $\phi \in \IL$, then $\phi \in \Eprimevalid$ (the extra freedom of $\Dtenprime$ granted to $\proponent$ compared to $\Dten$ need not be exercised).  Moreover, all instances of the stability scheme are, by construction, in $\Eprimevalid$, so $\Eprimevalid$ properly extends $\IL$.

We have so far not be able to problem of showing the converse of Theorem~\ref{stab-soundness}, but we conjecture that our proposed characterization of $\Stable$ is indeed correct.  To complete the proof that $\Eprimevalid$ is $\Stable$, it would suffice to show that~(i)~$\Eprimevalid$ is closed under modus ponens, and~(ii)~$\Eprimevalid$ is closed under uniform substitution.  (Such an approach is grounded on a Hilbert-style approach to stable logic.)  Another route would be to try to go in the opposite direction taken in the proof of Lemma~\ref{lemma:lem-1}.  There, we defined the notion of ``application'' of $\Dtenprime$ and showed how to eliminate them by explicitly postulating instances of the stability principle and manipulating the way the winning strategy starts.  Ideally, one would want to go the other way: show how, from an $\Eruleset$-winning strategy for
\[
\left ( \bigwedge_{1 \leq i \leq n} \neg\neg p_{i} \rightarrow p_{i} \right ) \rightarrow \phi,
\]
to construct a $\Dtenstar$-winning strategy for $\phi$ alone by introducing (rather than eliminating) what we have called applications of $\Dtenstar$.  The chief difficulty is to characterize the possible ways that a stability hypothesis $\neg\neg p \rightarrow p$ is used.

\subsection{Independent rulesets}\label{sec:independent}

Felscher indicates that rule $\Erule$ implies $\Dthirteen$ and, for
odd positions, $\Deleven$ and $\Dtwelve$, too.  This means that every
dialogue that adheres to rule $\Erule$ also adheres to $\Dthirteen$,
and if we understand $\Deleven$ and $\Dtwelve$ as quantifying over
move positions $0$, $1$, $\dots$, then every dialogue that adheres to
the $\Erule$ also adheres to $\Deleven$ and $\Dtwelve$ if the
quantifiers in these rules are restricted to odd numbers.

The fact that standard dialogue rules can imply each other, wholly or
partially, is an obstacle for solving the composition problem for
subsets of standard rulesets.  What we would seek are
\emph{independent} sets of dialogue rules, that is, sets of rules each
member of which is not implied by the others.

The examples above of $\N$ and the failed dialogical characterization
of $\LQ$ demonstrate the sensitive dependence of $\setofvalid{S}$ on a
set $S$ of dialogue rules.  Slight changes to a set $S$ of dialogue
rules can cause $\setofvalid{S}$ to shift from being a familiar logic
to a curiosity like $\N$ or the result of the failed characterization
of $\LQ$ (which may not even be a logic at all, in the sense of not
being closed under modus ponens)  On the other hand, sometimes simple semantic intuitions do apparently lead to success, as in the case of stable logic in Section~\ref{sec:stable}.

The demonstrated sensitivity may turn out to be an intrinsic feature
of the dialogical approach.  Moreover, sensitivity can be found
outside dialogues, too: one can jump from intuitionistic to classical
logic in a Hilbert-style calculus---an enormous leap, from the point
of view of the lattice of intermediate logics---in a single step by
adding a single new axiom (e.g., excluded middle or Peirce's formula).
And one can move from intuitionistic to classical logic by simply
dropping a constraint on the number of formulas that can appear on the
right-hand side of a sequent.\footnote{The claim here is not that all
  Hilbert-style and sequent calculi for intuitionistic logic are such
  that adding one new principle or dropping exactly one structural
  condition are sufficient to capture classical logic; there are
  precise calculi for which these claims hold.  A concrete example of
  a suitable Hilbert-style calculus is provided by the axioms $B$,
  $C$, $K$, and $I$, with Peirce's formula~\cite{seldin1997}; the
  calculus $G2$~\cite{troelstra-schwichtenberg} is a suitable example
  of a sequent calculus.}

Nonetheless, the non-independence of standard sets of dialogue rules
is an obstacle to solving both the uniform and non-uniform
composition problems.  From a foundation of an independent set of
dialogue rules, the problem of exposing some structure becomes easier
because we can gradually add or subtract rules with the confidence
that we are not making impermissible ``jumps'' in the space of
possibilities.

\section{Conclusion}\label{sec:conclusion}

Dialogue games can be viewed either as a stylized form of rational
interaction or as alternative logical calculi.  We have raised two
problems---the modus ponens problem and the uniform substitution
problem---that, on either view, pose challenges for the dialogician
that are, so far, largely unaddressed.  It seems that today much
remains to be done for dialogues to give them their proper
proof-theoretic foundation.  We have argued, more precisely, that one
important gap is that we lack a structure theory for dialogues that
could help shed light on the problem of precisely what is the force of
various dialogue rules.

We are not able to give a precise characterization of ``direct''
solutions to problems in dialogical logic.  We gave a few results and
conjectures (see, e.g., Lemma~\ref{no-explosion}) along the intended
lines.  we in fact do not yet possess a direct solution.  We left open
the problem of showing, directly, that classical propositional logic
is consistent, in the sense that there is no $\CLruleset$-valid
formula $\phi$ such that $\CLvalid{\phi}$ but satisfies $\CLvalid{\phi
  \rightarrow \psi}$ for all formulas $\psi$.  It seems plausible to us that a direct solution to
the problem is available; the solution may go via a normal form
theorem for dialogue games.

One reason behind our preference for direct solutions to problems in
dialogical logic is to stimulate the development of the dialogue
formalism so that it becomes more systematic.  In our view, resorting
to external devices to establish basic results in dialogical logic is
not a vote of confidence for dialogues, but is implicitly a concession
that the formalism is awkward and difficult to work with.

Returning to the non-winning play that motivated our investigation (see Table~\ref{stable-non-win}, another option for proceeding would be to drop rule $\Deleven$ but keep the other rules.  This opens the door to $\proponent$ having a winning strategy, but it is not clear what the logical characterization is.  More precisely: A known characterization of $\CL$ is the $\Eruleset$ minus rules $\Deleven$ and $\Dtwelve$; it is possible that dropping $\Deleven$, but keeping $\Dtwelve$, gives us stable logic or at least a logic in which stability is provable.  We leave this as an open problem.  We also leave open the problem of showing that $\Eprimevalid$ is closed under uniform substitution and modus ponens.  Solving both of these two problems would, combined with our soundness result, entail that the simple modification we made to the structural rules for intuitionistic logic do indeed yield stable logic.

\bibliographystyle{splncs03}
\bibliography{alama}

\end{document}